\documentclass[oneside,english,american]{amsart}
\usepackage[T1]{fontenc}
\usepackage[latin9]{inputenc}
\usepackage{refstyle}
\usepackage{amsthm}
\usepackage{amstext}
\usepackage{amssymb}

\makeatletter

\AtBeginDocument{\providecommand\propref[1]{\ref{prop:#1}}}
\RS@ifundefined{subref}
  {\def\RSsubtxt{section~}\newref{sub}{name = \RSsubtxt}}
  {}
\RS@ifundefined{thmref}
  {\def\RSthmtxt{theorem~}\newref{thm}{name = \RSthmtxt}}
  {}
\RS@ifundefined{lemref}
  {\def\RSlemtxt{lemma~}\newref{lem}{name = \RSlemtxt}}
  {}

\numberwithin{equation}{section}
\numberwithin{figure}{section}
\theoremstyle{plain}
\newtheorem{thm}{\protect\theoremname}
  \theoremstyle{plain}
  \newtheorem{prop}[thm]{\protect\propositionname}
  \theoremstyle{plain}
  \newtheorem{cor}[thm]{\protect\corollaryname}
  \theoremstyle{definition}
  \newtheorem{defn}[thm]{\protect\definitionname}
  \theoremstyle{remark}
  \newtheorem{claim}[thm]{\protect\claimname}

\date{}

\makeatother

\usepackage{babel}
  \addto\captionsamerican{\renewcommand{\claimname}{Claim}}
  \addto\captionsamerican{\renewcommand{\corollaryname}{Corollary}}
  \addto\captionsamerican{\renewcommand{\definitionname}{Definition}}
  \addto\captionsamerican{\renewcommand{\propositionname}{Proposition}}
  \addto\captionsamerican{\renewcommand{\theoremname}{Theorem}}
  \addto\captionsenglish{\renewcommand{\claimname}{Claim}}
  \addto\captionsenglish{\renewcommand{\corollaryname}{Corollary}}
  \addto\captionsenglish{\renewcommand{\definitionname}{Definition}}
  \addto\captionsenglish{\renewcommand{\propositionname}{Proposition}}
  \addto\captionsenglish{\renewcommand{\theoremname}{Theorem}}
  \providecommand{\claimname}{Claim}
  \providecommand{\corollaryname}{Corollary}
  \providecommand{\definitionname}{Definition}
  \providecommand{\propositionname}{Proposition}
\providecommand{\theoremname}{Theorem}

\begin{document}

\title{On maximizing the speed of a random walk in fixed environments}

\author{Amichai Lampert \and Assaf Shapira}
\begin{abstract}
We consider a random walk in a fixed $\mathbb{Z}$ environment composed
of two point types: $\left(q,1-q\right)$ and $\left(p,1-p\right)$
for $\frac{1}{2}<q<p$. We study the expected hitting time at $N$
for a given number $k$ of $p$-drifts in the interval $[1,N-1]$,
and find that this time is minimized asymptotically by equally spaced
$p$-drifts.
\end{abstract}
\maketitle

\section{Introduction}

Procaccia and Rosenthal \cite{procaccia2012need} studied how to optimally
place given number of vertices with a positive drift on top of a simple
random walk to minimize the expected crossing time of an interval.
They ask about extending their work to the situation where the environment
on $\mathbb{Z}$ is composed of two point types: $\left(q,1-q\right)$
and $\left(p,1-p\right)$ for $\frac{1}{2}<q<p$. This is the goal
of this note. See \cite{procaccia2012need} for background and further
related work.

Consider nearest neighbor random walks on ${0,1,...,N}$ with reflection
at the origin. We denote the random walk by $\left\{ X_{n}\right\} _{n=0}^{\infty}$
, and by $\omega\left(i\right)$ the transition probability at vertex
$i$:

\begin{eqnarray*}
P(X_{n+1}=i+1|X_{n}=i) & = & \omega\left(i\right)\\
P(X_{n+1}=i-1|X_{n}=i) & = & 1-\omega\left(i\right).
\end{eqnarray*}

First, we prove the following proposition concerning the expected
hitting time at vertex $N$:
\begin{prop}
\label{prop:hitting_time}For a walk $\omega$ starting at $x$, the
hitting time $T_{N}=\min\left\{ n\ge0|X_{n}=N\right\} $ satisfies:

\[
E_{\omega}^{x}\left(T_{N}\right)=N-x+2\sum_{i=x}^{N-1}\sum_{j=1}^{i}\prod_{k=j}^{i}\rho_{k},
\]

where $\rho_{i}=\frac{1-\omega\left(i\right)}{\omega\left(i\right)}$,
and $E_{\omega}^{x}\left(T_{N}\right)$ stands for the expected hitting
time. In particular:

\[
E_{\omega}^{0}\left(T_{N}\right)=N+2\sum_{i=1}^{N-1}\sum_{j=1}^{i}\prod_{k=j}^{i}\rho_{k}.
\]
\end{prop}
\begin{cor}
The expected hitting time from $0$ to $N$ is symmetric under reflection
of the environment, i.e. taking the environment $\omega^{\prime}\left(i\right)=\omega\left(N-i\right)$
gives $E_{\omega^{\prime}}^{0}\left(T_{N}\right)=E_{\omega}^{0}\left(T_{N}\right)$.
\end{cor}
Next we turn to the case of an environment consisting of two types
of drifts, $\left(q,1-q\right)$ (i.e. probability $q$ to go to the
right and $1-q$ to the left) and $\left(p,1-p\right)$, for some
$\frac{1}{2}<q<p\le1$. Assume that $k$ of the vertices are $p$-drifts,
and the rest are $q$-drifts. In \cite{procaccia2012need} it was
proven that for $q=\frac{1}{2}$ equally spaced $p$-drifts minimize
$\frac{E_{\omega}^{0}\left(T_{N}\right)}{N}$ (for large $N$). In
this paper we extend this result for $q>\frac{1}{2}$. We define an
environment in which the $p$-drifts are equally spaced (up to integer
effects):

\[
\omega_{N,k}\left(x\right)=\begin{cases}
p & x=\left\lfloor i\cdot\frac{N-1}{k}\right\rfloor \,\text{for some}\,\,1\le i\le k\\
q & \text{otherwise}
\end{cases},
\]

and prove the following theorem:
\begin{thm}
\label{thm:kn}For every $\varepsilon>0$ there exists $n_{0}$ such
that for every $N>n_{0}$ and environment $\omega$:

\[
\frac{E_{\omega}^{0}\left(T_{N}\right)}{N}>\frac{E_{\omega_{N,k}}^{0}\left(T_{N}\right)}{N}-\varepsilon,
\]

where $k$ is the number of $p$-drifts in $\omega$.
\end{thm}
Finally, we consider the set of environments $\omega_{ak,k}$ for
some $a\in\mathbb{N}$, and calculate $\lim\limits _{k\rightarrow\infty}\frac{E_{\omega_{ak,k}}^{0}\left(T_{N}\right)}{ak}$
:
\begin{prop}
\label{prop:calculation}Let $a\in\mathbb{N}$. Then:

\[
\lim_{k\rightarrow\infty}\frac{E_{\omega_{ak,k}}^{0}\left(T_{ak}\right)}{ak}=1+\frac{2}{a}\cdot\frac{\alpha^{a+2}-a\alpha^{3}+\left(a-1\right)\alpha^{2}+\left(\left(a\alpha^{2}-\left(a+1\right)\alpha\right)\alpha^{a}+\alpha\right)\beta}{\left(\alpha^{2}-2\alpha+1\right)\alpha^{a}\beta-\alpha^{3}+2\alpha^{2}-\alpha}.
\]

\end{prop}

\section{Proof of the main theorem}
\begin{proof}[Proof of Proposition \propref{hitting_time} ]
Define $v_{x}=E_{\omega}^{x}\left(T_{N}\right)$ for $0\le x\le N$.
By conditioning on the first step:
\begin{enumerate}
\item $v_{N}=0$
\item $v_{0}=v_{1}+1$
\item $v_{x}=p_{x}v_{x+1}+\left(1-p_{x}\right)v_{x-1}+1\qquad1\le x\le N-1.$
\end{enumerate}

To solve these equations, define $a_{x}=v_{x}-v_{x-1}$ (for $1\le x\le N$)
and $b_{x}=v_{x+1}-v_{x-1}$ (for $1\le x\le N-1$). Then:

\begin{eqnarray*}
b_{x} & = & a_{x}+a_{x+1}\\
a_{x} & = & p_{x}b_{x}+1\\
a_{1} & = & -1
\end{eqnarray*}

We get for $a_{x}$ the relation $a_{x+1}=\rho_{x}a_{x}-\rho_{x}-1$,
whose solution is $a_{x}=-2\sum\limits _{j=1}^{x-1}\prod\limits _{k=j}^{x-1}\rho_{k}-1$,
and then:

\begin{eqnarray*}
v_{x} & = & \sum_{i=x+1}^{N}\left(v_{i-1}-v_{i}\right)+v_{N}\\
 & = & \sum_{i=x+1}^{N}\left(-a_{i}\right)+v_{N}\\
 & = & N-x+2\sum_{i=x}^{N-1}\sum_{j=1}^{i}\prod_{k=j}^{i}\rho_{k}
\end{eqnarray*}
\end{proof}
\begin{defn}
To evaluate $E_{\omega}^{0}\left(T_{N}\right)$ we define:

\begin{eqnarray*}
S_{N} & = & \sum_{i=1}^{N-1}\sum_{j=1}^{i}\prod_{k=j}^{i}\rho_{k}=\sum_{d=1}^{N-1}\sum_{j=1}^{N-d}\prod_{k=j}^{j+d-1}\rho_{k}.
\end{eqnarray*}

Next define $\widetilde{\rho}_{k}$ for $k$ in the circle $\mathbb{Z}_{N-1}$,
such that for $1\le k\le N-1$ we will have $\widetilde{\rho}_{k}=\rho_{k}$
(gluing the point $0$ to the point $N-1$), and then look at:

\[
\widetilde{S}_{N}=\sum_{d=1}^{N-1}\sum_{j=1}^{N-1}\prod_{k=j}^{j+d-1}\widetilde{\rho}_{k}.
\]

This way, rather than summing $\prod\limits _{k=i}^{j}\rho_{k}$ over
subintervals $\left[i,j\right]$ of $[1,N-1]$, we sum \foreignlanguage{english}{$\prod\limits _{k=i}^{j}\widetilde{\rho}_{k}$}
over subintervals of the circle $\mathbb{Z}_{N-1}$.\end{defn}
\begin{prop}
\label{prsn:circle_approximation}Define $\alpha=\frac{1-q}{q},\,\beta=\frac{1-p}{p}$.
Since $\beta<\alpha<1$:
\end{prop}
\begin{eqnarray*}
\left|\widetilde{S}_{N}-S_{N}\right| & = & \sum_{d=1}^{N-1}\sum_{j=N-d+1}^{N-1}\prod_{k=j}^{j+d-1}\rho_{k}\\
 & \le & \sum_{d=1}^{N-1}d\alpha^{d}\\
 & \le & \sum_{d=1}^{\infty}d\alpha^{d}<C(\alpha)
\end{eqnarray*}

for some constant $C\left(\alpha\right)$ which doesn't depend on
$N$.
\begin{defn}
Let $n_{i}^{\left(d\right)}$ be the number of $p$-drifts in the
interval $\left[i,i+d-1\right]$.
\end{defn}
Since every drift appears in $d$ intervals of length $d$, $\sum\limits _{i=1}^{N-1}n_{i}^{\left(d\right)}=dk$.
Also,

\begin{eqnarray*}
\widetilde{S}_{N} & = & \sum_{d=1}^{N-1}\sum_{i=1}^{N-1}\left(\frac{\beta}{\alpha}\right)^{n_{i}^{\left(d\right)}}\cdot\alpha^{d}\\
 & = & \sum_{d=1}^{N-1}\sigma_{d}
\end{eqnarray*}

where $\sigma_{d}=\sum\limits _{i=1}^{N-1}\left(\frac{\beta}{\alpha}\right)^{n_{i}^{\left(d\right)}}\cdot\alpha^{d}$.
\begin{claim}
\label{clm:min_sigmad}For $n_{l}^{\left(d\right)}\in\mathbb{N}$
the expression $\sigma_{d}$ is minimized under the restriction $\sum\limits _{l=1}^{N-1}n_{l}^{\left(d\right)}=dk$
if $n_{i}^{\left(d\right)}-n_{j}^{\left(d\right)}\le1$ for all $i,j$.\end{claim}
\begin{proof}
For convenience, we omit $d$ from the notation, and set $\mathbf{n}=\left(n_{1},...,n_{N-1}\right)$.
If a vector $\mathbf{n}$ satisfies $n_{i}-n_{j}\le1\,\forall i,j$,
we say $\mathbf{n}$ is almost constant. We will show that $\sigma$
is minimal for some almost constant vector. Then we show that $\sigma$
takes on the same value for all almost constant vectors under the
restriction, and this completes the proof.

Suppose $\sigma$ is minimized (under the restriction) by some vector
$\mathbf{n}^{0}$. If $\mathbf{n}^{0}$ is almost constant, we are
done. Else, for some $i,j$ we have that $n_{i}^{0}-n_{j}^{0}\ge2$.
We choose $i,j$ such that \foreignlanguage{english}{$n_{i}^{0}-n_{j}^{0}$}
is maximal. Define:

\[
n_{l}^{1}=\begin{cases}
n_{l}^{0} & l\ne i,j\\
n_{l}^{0}-1 & l=i\\
n_{l}^{0}+1 & l=j
\end{cases}.
\]

$\mathbf{n}^{1}$ satisfies the restriction, and $\sigma\left(\mathbf{n}^{0}\right)\ge\sigma\left(\mathbf{n}^{1}\right)$:

\begin{eqnarray*}
\sigma\left(\mathbf{n}^{0}\right)-\sigma\left(\mathbf{n}^{1}\right) & = & \sum\limits _{t=1}^{N-1}\left(\frac{\beta}{\alpha}\right)^{n_{t}^{0}}\cdot\alpha^{d}-\sum\limits _{t=1}^{N-1}\left(\frac{\beta}{\alpha}\right)^{n_{t}^{1}}\cdot\alpha^{d}\\
 & = & \alpha^{d}\left(\left(\frac{\beta}{\alpha}\right)^{n_{i}^{0}}+\left(\frac{\beta}{\alpha}\right)^{n_{j}^{0}}-\left(\frac{\beta}{\alpha}\right)^{n_{i}^{0}-1}-\left(\frac{\beta}{\alpha}\right)^{n_{j}^{0}+1}\right)\\
 & = & \alpha^{d}\left(1-\frac{\beta}{\alpha}\right)\left(\left(\frac{\beta}{\alpha}\right)^{n_{j}^{0}}-\left(\frac{\beta}{\alpha}\right)^{n_{i}^{0}-1}\right)\\
 & \ge & 0,
\end{eqnarray*}

where the inequality follows from the fact that $0\le\frac{\beta}{\alpha}<1$
and $n_{j}^{0}<n_{i}^{0}-1$. From minimality of $\sigma\left(\mathbf{n}^{0}\right)$,
we get that $\sigma\left(\mathbf{n}^{1}\right)$ is also minimal.
This process must end after a finite number of steps $f$, yielding
an almost constant $\mathbf{n}^{f}$ which minimizes $\sigma$.

Now for a general almost constant vector $\mathbf{n}$, set $a=\min\left\{ n_{l}:\,1\le l\le N-1\right\} $.
We have $n_{l}\in\left\{ a,a+1\right\} $, so defining $m_{0}$ to
be the number of $a$'s and $m_{1}$ to be the number of $a+1$'s,
we get:

\begin{eqnarray*}
dk & = & \sum\limits _{l=1}^{N-1}n_{l}\\
 & = & m_{0}a+m_{1}\left(a+1\right)\\
 & = & \left(m_{0}+m_{1}\right)a+m_{1}\\
 & = & \left(N-1\right)a+m_{1},
\end{eqnarray*}

and since $m_{1}<N-1$, there is a unique solution for natural $a,m_{1}$.
So all almost constant $\mathbf{n}$ (satisfying the restriction)
are the same up to ordering, and since $\sigma$ doesn't depend on
the order, they all give the same value.\end{proof}
\begin{claim}
\label{clm:drift_placement}For every choice of $M,k$, the placement
of $k$ drifts on the circle $\mathbb{Z}_{M}$ in which the $i$th
drift is at the point $\left\lfloor i\cdot\frac{M}{k}\right\rfloor $
satisfies:

\[
\forall d,i,j\quad n_{i}^{\left(d\right)}-n_{j}^{\left(d\right)}\le1.
\]
\end{claim}
\begin{proof}
Place the $i$th drift at the point $\left\lfloor i\cdot\frac{M}{k}\right\rfloor $.
We calculate the number of drifts in the interval $\left[x,x+d-1\right]$.
The first drift inside this interval is:

\begin{eqnarray*}
\left\lfloor i_{0}\cdot\frac{M}{k}\right\rfloor  & \ge & x\\
i_{0}\cdot\frac{M}{k} & \ge & x\\
i_{0} & \ge & x\cdot\frac{k}{M}\\
i_{0} & = & \left\lceil x\cdot\frac{k}{M}\right\rceil .
\end{eqnarray*}

The last drift inside this interval is:

\begin{eqnarray*}
\left\lfloor i_{1}\cdot\frac{M}{k}\right\rfloor  & \le & x+d-1\\
i_{1}\cdot\frac{M}{k} & < & x+d\\
i_{1} & < & \left(x+d\right)\cdot\frac{k}{M}\\
i_{1} & = & \left\lceil \left(x+d\right)\cdot\frac{k}{M}\right\rceil -1.
\end{eqnarray*}

The number of drifts inside this interval is therefore:

\begin{eqnarray*}
i_{1}-i_{0}+1 & = & \left\lceil \left(x+d\right)\cdot\frac{k}{M}\right\rceil -\left\lceil x\cdot\frac{k}{M}\right\rceil \\
 & \ge & \left(x+d\right)\cdot\frac{k}{M}-x\cdot\frac{k}{M}-1\\
 & = & \frac{dk}{M}-1\\
i_{1}-i_{0}+1 & \le & \left(x+d\right)\cdot\frac{k}{M}+1-x\cdot\frac{k}{M}\\
 & = & \frac{dk}{M}+1.
\end{eqnarray*}

So for non-integer $\frac{dk}{M}$ the number of drifts takes on only
the two values $\left\lfloor \frac{dk}{M}\right\rfloor ,\left\lceil \frac{dk}{M}\right\rceil $.
For integer $\frac{dk}{M}$ we simply have:

\begin{eqnarray*}
i_{1}-i_{0}+1 & = & \left\lceil \left(x+d\right)\cdot\frac{k}{M}\right\rceil -\left\lceil x\cdot\frac{k}{M}\right\rceil \\
 & = & \frac{dk}{M}
\end{eqnarray*}
\end{proof}
\begin{claim}
$\widetilde{S}_{N}$ is minimal for the configuration of drifts described
by $\omega_{N,k}$ (where the $i$th drift is at vertex $\left\lfloor i\cdot\frac{N-1}{k}\right\rfloor $).\end{claim}
\begin{proof}
$\widetilde{S}_{N}=\sum\limits _{d=1}^{N-1}\sigma_{d}$, and by claims
\ref{clm:min_sigmad} and \ref{clm:drift_placement} each $\sigma_{d}$
is minimized by this configuration, therefore the sum is also minimized.
\end{proof}

\begin{proof}[Proof of Theorem \ref{thm:kn}]
From Proposition \ref{prsn:circle_approximation}, $0<\widetilde{S}_{N}-S_{N}<C$.
Let $n_{0}=\frac{2C}{\varepsilon}$. Then for $N>n_{0}$:

\begin{eqnarray*}
\frac{E_{\omega}^{0}\left(T_{N}\right)}{N} & = & \frac{N+2S_{N}}{N}\\
 & = & 1+2\frac{S_{N}}{N}\\
 & > & 1+2\frac{\widetilde{S}_{N}}{N}-\varepsilon\\
 & \ge & 1+2\frac{\widetilde{S}_{N}^{*}}{N}-\varepsilon\\
 & \ge & 1+2\frac{S_{N}^{*}}{N}-\varepsilon\\
 & = & \frac{E_{\omega_{N,k}}^{0}\left(T_{N}\right)}{N}-\varepsilon
\end{eqnarray*}

where we denote by $S_{N}^{*}$ and $\widetilde{S}_{N}^{*}$ the values
caculated for $\omega_{N,k}$.
\end{proof}

\begin{proof}[Proof of Proposition \ref{prop:calculation}]
We evaluate $\lim\limits _{k\rightarrow\infty}\frac{\widetilde{S}_{ak}^{*}}{ak}$.
First, we consider the $k$ intervals that do not contain any $\beta$,
each of which contributes:

\[
s_{0}=\sum_{i=1}^{a-1}\left(a-i\right)\alpha^{i}.
\]

Next we consider the $k$ intervals that contain $n\ge1$ $\beta$'s:

\[
s_{n}=\beta^{n}\cdot\alpha^{\left(a-1\right)\left(n-1\right)}\cdot\sum_{r=0}^{a-1}\sum_{s=0}^{a-1}\alpha^{r+s}.
\]

Then we get:

\begin{eqnarray*}
\lim\limits _{k\rightarrow\infty}\frac{\widetilde{S}_{ak}^{*}}{ak} & = & \frac{1}{a}\lim\limits _{k\rightarrow\infty}\frac{ks_{0}+\sum\limits _{n=1}^{k}ks_{n}}{k}\\
 & = & \frac{1}{a}\cdot\frac{\alpha^{a+2}-a\alpha^{3}+\left(a-1\right)\alpha^{2}+\left(\left(a\alpha^{2}-\left(a+1\right)\alpha\right)\alpha^{a}+\alpha\right)\beta}{\left(\alpha^{2}-2\alpha+1\right)\alpha^{a}\beta-\alpha^{3}+2\alpha^{2}-\alpha},
\end{eqnarray*}

and since $\lim\limits _{k\rightarrow\infty}\frac{\widetilde{S}_{ak}^{*}-S_{ak}^{*}}{ak}=0$
from Proposition \ref{prsn:circle_approximation}, the proof is complete.
\end{proof}

\section{Further questions}
\begin{enumerate}
\item Show that the optimal environment also minimizes the variance of the
hitting time.
\item Can this result be extended to a random walk on $\mathbb{Z}$ with
a given density of drifts (as in \cite{procaccia2012need})?
\item Can similar results be found for other graphs? For example, $\mathbb{Z}_{2}\times\mathbb{Z}_{N}$.
\end{enumerate}
\textbf{Acknowledgements:} Thanks to Eviatar Procaccia and Itai Benjamini
for introducing us to this problem and for useful discussions. 

\bibliographystyle{amsplain}
\bibliography{need_for_speed}

\end{document}